\documentclass[12pt]{article}

\usepackage{amssymb}%
\usepackage{amsmath}%
\usepackage{amsthm}%
\usepackage{hyperref}%
\usepackage{graphicx}%
\usepackage{xcolor}%
\usepackage{mathrsfs}

\allowdisplaybreaks%

{\theoremstyle{plain}%
 \newtheorem{theorem}{Theorem}
 
 \newtheorem{lemma}{Lemma}% 
}
{\theoremstyle{remark}

}
{\theoremstyle{definition}

}

\DeclareRobustCommand{\stirling}{\genfrac\{\}{0pt}{}}

\begin{document}

\begin{center}
{\large On a problem on a generalization of Euler's totient function}

 \ 
 
{\sc John M. Campbell}

\vspace{0.1in}

{\footnotesize Department of Mathematics and Statistics}

{\footnotesize Dalhousie University}

{\footnotesize Halifax, NS B3H 4R2}

{\footnotesize Canada}

\vspace{0.1in}

{\footnotesize {\tt jh241966@dal.ca}}

\vspace{0.1in}

\end{center}

\begin{abstract}
 B\"uy\"uka\c sik et al.\ [Publ.\ Math.\ Debrecen, 2024] introduced a family of generalizations of Euler's totient function $\varphi(n)$, by 
 setting $\varphi_k(n) = \sum_{a} a^k$ for $a \in [1, n]$ such that $(a, n) = 1$, with $\varphi_0(n) = \varphi(n)$. Letting $\mathcal{D}_{s} = 
 \{ k \geq s : \forall n \geq 1 \ \varphi_s(n) \mid \varphi_k(n) \}$, B\"uy\"uka\c sik et al.\ proved that $\mathcal{D}_{s}$ is finite for each 
 $s \geq 0$, and conjectured that $\mathcal{D}_{1} = \{ 1, 3, 15 \}$ and provided computations to support this conjecture. 
 We succeed in proving this conjecture, using an argument based on our extensive interactions with GPT-5.5 Pro. 
\end{abstract}

\vspace{0.1in}


\vspace{0.1in}


\section{Introduction}
 Euler's totient function $\varphi(n)$ giving the number of integers $k \in [1, n]$ such that $k$ and $n$ are coprime is one of the most 
 fundamental arithmetic functions. The importance of the Euler totient function inspires research on generalizations of the form
\begin{equation}\label{generalizephi} 
 \varphi_{k}(n) = \sum_{\substack{ a \in [1, n] \\ (a, n) = 1 }} a^k, 
\end{equation}
 as in the work of B\"uy\"uka\c sik et al.\ \cite{BuyukasikGoralSertbas2024}, who built upon the previous work of Singh \cite{Singh2009} 
 concerning \eqref{generalizephi}. Functions of the form shown in \eqref{generalizephi} are noted in classic texts in both advanced 
 combinatorics \cite[p.\ 199]{Comtet1974} and analytic number theory \cite[p.\ 48]{Apostol1976}. In this latter text, Apostol noted 
 the relations among 
\begin{align}
 \sum_{d \mid n} \frac{ \varphi_k(d) }{d^k} & = \frac{1^k + 2^k + \cdots + n^k}{n^k}, \nonumber \\ 
 \varphi_1(n) & = \frac{1}{2} n \varphi(n) \ \text{for $n > 1$, and} \label{phi1eval} \\
 \varphi_2(n) & = \frac{1}{3} n^2 \varphi(n) + \frac{n}{6} \prod_{p \mid n} (1-p) \ \text{for $n > 1$.} \nonumber 
\end{align}
 In this paper, we prove a conjecture from B\"uy\"uka\c sik et al.\ concerning \eqref{generalizephi}. It appears that this conjecture has 
 remained open, prior to our work. 

 Adopting notation from B\"uy\"uka\c sik et al.\ \cite{BuyukasikGoralSertbas2024}, we write $$ \mathcal{D}_{s} = \{ k \geq s : \forall n 
 \geq 1 \ \varphi_{s}(n) \mid \varphi_k(n) \}, $$ for a fixed and nonnegative integer $s$. B\"uy\"uka\c sik et al.\ proved that 
 $\mathcal{D}_{s}$ is finite for any fixed integer $s \geq 0$, 
 and conjectured that $\mathcal{D}_{1} = \{ 1, 3, 15 \}$
 and provided computations to support this conjecture. We prove this conjecture in the below section, 
 using an argument derived from our many interactions with GPT-5.5 Pro. 

\section{A full solution}
 Since B\"{u}y\"{u}ka\c{s}ik et al.\ \cite[Theorem 2(2)]{BuyukasikGoralSertbas2024} proved that $ \{ 1, 3, 15 \} \subseteq \mathcal{D}_{1}$ 
 (noting that it is immediate that $s \in \mathcal{D}_{s}$), it remains to prove the reverse inclusion. To begin with, we recall the definition 
 of the sequence of Bernoulli numbers, which may be defined via the generating function relation $ \frac{x}{e^{x} - 1} = \sum_{n = 
 0}^{\infty} \frac{B_{n} x^n}{n!}$. We also require the use of the von Staudt--Clausen theorem, giving us that: For every even integer $m 
 \geq 2$, the denominator of $B_{m}$ is $\prod_{(p-1) \mid m} p$. 

 Letting $k > 1$ be an odd integer, and letting $v$ be a positive integer, a key to our construction is given by Bernoulli sums of the form 
\begin{equation}\label{defineCku} 
 C_{k}(v) = \frac{2}{k+1} \sum_{j=0}^{(k-1)/2} \binom{k+1}{2j} B_{2j} (1 - 2 j)^{v}. 
\end{equation}
 Similarly, \emph{Bernoulli polynomials} may be defined so that 
\begin{equation}\label{defineBNx} 
 B_{N}(x) = \sum_{r=0}^{N} \binom{N}{r} B_{r} x^{N-r}
\end{equation}
 and are required for the purposes of our proof of the following key lemma. Our construction also requires the use of \emph{Stirling 
 numbers of the second kind}, which may be defined so that $\stirling{n}{k} = \frac{1}{k!} \sum_{j=0}^{k} (-1)^{k-j} \binom{k}{j} j^n$. 
 By writing $D$ in place of the differential operator $x \frac{d}{dx}$, the relation 
\begin{equation}\label{DStirling} 
 D^{m} = \sum_{\ell = 1}^{m} \stirling{m}{\ell} x^{\ell} \frac{d^{\ell}}{dx^{\ell}} 
\end{equation}
 is required for our purposes and can be found  (in an equivalent form) 
  in Comtet's text \cite[p.\ 220]{Comtet1974}. 

\begin{lemma}\label{triangularlemma}
 For $k > 1$ odd, suppose that $C_k(v)$ is an integer for every $v \geq 1$. Then, for every even integer $h \in [2, k-1]$, 
 we have that 
 $ 2 \frac{k!}{h!} B_{h} \in \mathbb{Z}$. 
\end{lemma}

\begin{proof}
 Write $N = k + 1$, with $N \geq 4$ being even. For an integer 
\begin{equation}\label{srange} 
 s \in \left[ 1, \frac{N-2}{2} \right], 
\end{equation}
 we define
\begin{equation}\label{defineTs}
 T_{N, s} = T_{s} = 2 \frac{(N-1)!}{(N-2s)!} B_{N-2s}. 
\end{equation}
 Since every even integer $h \in [2, k-1]$ can be written uniquely in the form
 $h = N - 2s$ for $N$ and $s$ as specified, 
 it remains to prove that $T_{s} \in \mathbb{Z}$. 

 Now, define 
\begin{equation}\label{defineFN}
 F_{N}(x) = x^{1-N} \big( B_{N}(x) - B_{N} \big), 
\end{equation}
 recalling the definition of Bernoulli polynomials in \eqref{defineBNx}. We claim that 
\begin{equation}\label{nestedD}
 C_{N - 1}(2s) = \frac{2}{N} \big( D^{2s} F_{N} \big)(1). 
\end{equation}
 To prove the relation in \eqref{nestedD}, we begin by manipulating the formula in 
 \eqref{defineBNx} to obtain that 
\begin{equation}\label{manipulatedBNx} 
 x^{1-N} B_{N}(x) = \sum_{r=0}^{N} \binom{N}{r} B_{r} x^{1-r},
\end{equation}
 so that \eqref{manipulatedBNx} gives us that 
\begin{equation}\label{sumtoreduce} 
 D^{2s}\big( x^{1-N} B_{N}(x) \big) \Big|_{x=1}
 = \sum_{r=0}^{N} \binom{N}{r} B_{r}(1-r)^{2s}. 
\end{equation}
 From the vanishing of the $r = 1$ case of the summand on the right of \eqref{sumtoreduce} together with the vanishing of Bernoulli 
 numbers with odd indices greater than $1$, we find that 
\begin{equation}\label{combinewithFN} 
 D^{2s}\big( x^{1-N} B_{N}(x) \big) \Big|_{x=1}
 = \frac{N}{2} C_{N-1}(2s) + B_{N}(1-N)^{2s}, 
\end{equation}
 so that an application of the definition of $F_{N}(x)$ in \eqref{defineFN}
 via \eqref{combinewithFN} gives us an equivalent version of \eqref{nestedD}. 

 Now, we rewrite $F_{N}(x)$, as defined in \eqref{defineFN}, by setting $g(x) = x^{1-N}$ and $H(x) = B_{N}(x) - B_{N}$, with $F_{N}(x) 
 = g(x) H(x)$. Observe that $B_{N}(1) = B_{N}$ and $H(1) = 0$.
 The product rule for derivatives then gives us that 
\begin{equation}\label{diffprod} 
 F_{N}^{(\ell)}(1) = \sum_{r=1}^{\ell} \binom{\ell}{r} g^{(\ell - r)}(1) B_{N}^{(r)}(1),
\end{equation}
 noting that we obtain the vanishing of $\binom{\ell}{r} g^{(\ell - r)}(1) H^{(r)}(1) $
 for the $r = 0$ case. 
 Observe that $g^{(\ell - r)}(1) \in \mathbb{Z}$, recalling the definition of $g(x)$ as an integer power of $x$. 
 From the definition of Bernoulli polynomials in \eqref{defineBNx}, one may verify that 
\begin{equation}\label{setx1} 
 B_{N}^{(r)}(x) = \frac{N!}{(N-r)!} B_{N - r}(x) 
\end{equation}
 for positive integers $r \leq N$, and we specialize \eqref{setx1} to 
\begin{equation}\label{specializeto1}
 B_{N}^{(r)}(1) = \frac{N!}{(N-r)!} B_{N - r}(1). 
\end{equation}
 According to the product rule expansion in \eqref{diffprod}, we have that sum in this product rule is restricted to indices $r \in [1, 
 \ell]$. In view of the identity in \eqref{nestedD}, which involves the application of $D^{2s}$ to $F_{N}(1)$, it is implicit that $\ell \leq 2 
 s$, and we recall the bounds on $s$ in \eqref{srange}, giving us that $2 \leq N - r$. Now, if $r$ is odd, then $N - r \geq 2$ is odd. 
 So, if $r$ is odd, then $B_{N-r}(1) = B_{N-r} = 0$. So, from \eqref{specializeto1}, we have that if $r$ is odd, 
 then $B_{N}^{(r)}(1)$ vanishes. 
 So, for an even value $r = 2i$, 
 we find that 
\begin{equation}\label{BntoT} 
 \frac{2}{N} B_{N}^{(2i)}(1) = T_{i}, 
\end{equation}
 recalling the definition in \eqref{defineTs}. An application of 
 the expression in \eqref{nestedD} of $C_{N-1}(2s)$ in terms of $\big( D^{2s}F_{N} \big)(1)$ 
 together with 
 the classical Stirling number identity in \eqref{DStirling}
 give  us that 
\begin{equation}\label{separateStirling} 
 \big( D^{2s} F_{N} \big)(1) = F_{N}^{(2s)}(1)
 + \sum_{\ell = 1}^{2s-1} \stirling{2s}{\ell} F_{N}^{(\ell)}(1). 
\end{equation}
 Through an application of \eqref{diffprod} to \eqref{separateStirling}, we obtain 
\begin{multline*}
 \frac{2}{N} \big( D^{2s} F_{N} \big)(1) = \sum_{r=1}^{2s} \binom{2s}{r} g^{(2s - r)}(1) \frac{2}{N} B_{N}^{(r)}(1) + \\ 
 \sum_{\ell = 1}^{2s-1} \stirling{2s}{\ell} \sum_{r=1}^{\ell} \binom{\ell}{r} g^{(\ell - r)}(1) \frac{2}{N} B_{N}^{(r)}(1), 
\end{multline*}
 and we rewrite this as 
\begin{multline*}
 \frac{2}{N} \big( D^{2s} F_{N} \big)(1) = \frac{2}{N} B_{N}^{(2s)}(1) + \sum_{r=1}^{2s-1} \binom{2s}{r} g^{(2s - r)}(1) \frac{2}{N} B_{N}^{(r)}(1) 
 + \\ 
 \sum_{\ell = 1}^{2s-1} \stirling{2s}{\ell} \sum_{r=1}^{\ell} \binom{\ell}{r} g^{(\ell - r)}(1) \frac{2}{N} B_{N}^{(r)}(1). 
\end{multline*}
 For indices $r$ involved in the above sums, we have that $1 \leq r \leq 2s \leq N-2$, 
 with $N - r \geq 2$. So, if $r$ is odd, then $N - r$ is odd and $\geq 3$, 
 with $B_{N-r}(1) = B_{N - r} = 0$, and hence
\begin{multline*}
 \frac{2}{N} \big( D^{2s} F_{N} \big)(1) = \frac{2}{N} B_{N}^{(2s)}(1) + \sum_{r=1}^{s-1} \binom{2s}{2r} g^{(2s - 2r)}(1) \frac{2}{N} B_{N}^{(2r)}(1) 
 + \\ 
 \sum_{\ell = 1}^{2s-1} \stirling{2s}{\ell} \sum_{ r = 1 }^{ \lfloor {\ell}/{2} \rfloor} \binom{\ell}{2r} g^{(\ell - 2r)}(1) \frac{2}{N} B_{N}^{(2r)}(1). 
\end{multline*}
 From the formula for $C_{N - 1}(2s)$ in \eqref{nestedD} together with the definition in \eqref{BntoT}, we find that there exist integers 
 $a_{s, r}$ such that 
\begin{equation}\label{inductionT} 
 T_{s} = C_{N-1}(2s) - \sum_{r=1}^{s-1} a_{s, r} T_{r}. 
\end{equation}
 For the $s = 1$ case of \eqref{inductionT}, we find that $C_{k}(2) = T_{1}$, and $C_{k}(2) $ is an integer by assumption. So, from the 
 assumption that $C_{k}(v)$ is an integer for every $v \geq 1$ and for $k > 1$ odd, and since $N \geq 4$ is even,  an inductive argument 
 gives us, from  \eqref{inductionT}, that $T_{s}$ is an integer for 
 $s$ in the interval indicated in \eqref{srange}, 
 as desired. 
\end{proof}

\begin{theorem}\label{maintheorem}
 The containment $\mathcal{D}_{1} \subseteq \{ 1, 3, 15 \}$ holds. 
\end{theorem}

\begin{proof}
 Let $k \in \mathcal{D}_{1}$. 
 To evaluate
 $$ \mathcal{D}_{1} = \{ k \geq 1 : \forall n \geq 1 \ \varphi_{1}(n) \mid \varphi_k(n) \}, $$ 
 we begin by noting that $\varphi_1(3) = 3$ and that $\varphi_k(3) = 1 + 2^k$. 
 So, in order for $\varphi_1(3) \mid \varphi_k(3)$ to hold, the integer $k$ is required to be odd. 
 So, the integer $k$ is henceforth assumed to be odd. 

 Now, if $k = 1$, then $k$ is in the desired containing set. So, we henceforth let $k > 1$, 
 (again while working under the assumption that $k \in \mathcal{D}_{1}$). 

 We proceed to make use of a standard formulation of the Faulhaber formula whereby 
 $$ \sum_{a=1}^{N} a^{k}
 = \frac{1}{k+1} \sum_{r=0}^{k} (-1)^{r} \binom{k+1}{r} B_{r} N^{k+1-r}. $$
 It was proved by Singh \cite[Theorem 3]{Singh2009} that 
\begin{equation}\label{inclusionex} 
 \varphi_{k}(n) = \sum_{d \mid n} \mu(d) d^{k} \sum_{b=1}^{n/d} b^k, 
\end{equation}
 and the relation in \eqref{inclusionex} can be obtained via the M\"{o}bius inversion formula. 
 Write $$ \alpha_j(n) = \sum_{d \mid n} \mu(d) d^{j-1} = \prod_{p \mid n} (1 - p^{j - 1}) $$
 for $n > 1$. 
 Singh also obtained that 
\begin{equation}\label{Singhalpha} 
 \varphi_{k}(n) = \frac{1}{k+1} \sum_{j = 0}^{k} (-1)^{j} \binom{k+1}{j} B_{j} \alpha_{j}(n) n^{k+1-j} 
\end{equation}
 as a consequence of \eqref{inclusionex}. 

 Now, let $n > 1$. Since $B_{j}$ vanishes for odd $j > 1$, and since $\alpha_1(n)$ also vanishes, 
 we obtain from \eqref{Singhalpha} that 
\begin{equation}\label{bisection}
 \varphi_{k}(n) = \frac{n^k \varphi(n)}{k+1} 
 + \frac{1}{k+1} \sum_{j=1}^{(k-1)/2} \binom{k+1}{2j} B_{2j} n^{k+1-2j} \prod_{p \mid n} (1-p^{2j-1}), 
\end{equation}
 recalling the assumption that $k > 1$ is odd. 

 Now, suppose that $n$ is squarefree and that $m$ is an odd and positive integer. An application of Euler's product formula for $\varphi$ 
 then gives us that 
\begin{equation}\label{prodgeom} 
 \prod_{p \mid n} (1 - p^m) = (-1)^{\omega(n)} \varphi(n) \prod_{p \mid n} (1 + p + \cdots + p^{m-1}). 
\end{equation}
 Dividing both sides of \eqref{bisection} by 
 $\varphi_{1}(n)$, the relations in 
 \eqref{phi1eval} and \eqref{prodgeom} then give us that 
\begin{multline*}
 \frac{\varphi_k(n)}{\varphi_1(n)} 
 = \\ \frac{2n^{k-1}}{k+1} + \frac{2 (-1)^{\omega(n)}}{k+1}
 \sum_{j=1}^{(k-1)/2} \binom{k+1}{2j} B_{2j} n^{k-2j} 
 \prod_{p \mid n} (1 + p + \cdots + p^{2j-2}), 
\end{multline*}
 again with the assumptions that $n$ is squarefree. 
 Since we are working under the assumption that $k \in \mathcal{D}_{1}$, 
 we have that the quotient $ \frac{\varphi_k(n)}{\varphi_1(n)} $ is an integer. 

 Again with the ongoing assumption that $k \in \mathcal{D}_{1}$, we claim that $C_{k}(u) \in \mathbb{Z}$ for each $u \geq 1$. In this 
 direction, we let $L$ denote the least common multiple of the denominators, in lowest terms, among all expressions of the form 
 $\frac{2}{k+1} \binom{k+1}{2j} B_{2j}$ for $0 \leq j \leq \frac{k-1}{2}$, noting that the $j=0$ case reduces to $\frac{2}{k+1}$. Dirichlet's 
 theorem on primes in arithmetic progressions then gives us that 
 there are infinitely many primes congruent to $1$ modulo $L$. 
 For fixed $u \geq 1$, choose $u$ distinct primes $p_1$, $p_2$, $\ldots$, $p_u$
 that are all congruent to $1$ modulo $L$, 
 and set $n = p_1 p_2 \cdots p_u$. 
 Then $n$ is squarefree, and $u = \omega(n)$, and 
\begin{equation}\label{n1modL}
 n \equiv 1 \bmod L. 
\end{equation}
 Moreover, the expansion of the product 
 $\prod_{p \mid n} (1 + p + \cdots + p^{m-1}) $ reveals that 
\begin{equation}\label{mtoomega}
 \prod_{p \mid n} (1 + p + \cdots + p^{m-1}) \equiv m^{\omega(n)} \pmod L. 
\end{equation}
 From \eqref{n1modL} and \eqref{mtoomega} together, we find that 
\begin{equation}\label{summandmod}
 (-1)^{\omega(n)} n^{k-2j} \prod_{p \mid n} (1 + p + \cdots + p^{2j-2}) \equiv (1-2j)^{\omega(n)} \pmod L, 
\end{equation} 
 from the construction of $L$, letting $ j \in \big[1, \frac{k-1}{2} \big]$. Now, we proceed to rewrite $ \frac{ \varphi_k(n) }{ \varphi_1(n) } 
 - C_k(\omega(n))$ as 
\begin{multline*}
 \frac{2n^{k-1}}{k+1} + \frac{2 }{k+1} \sum_{j=1}^{(k-1)/2} \binom{k+1}{2j} B_{2j} (-1)^{\omega(n)} n^{k-2j} 
 \prod_{p \mid n} (1 + p + \cdots + p^{2j-2}) - \\ 
 \frac{2}{k+1} \sum_{j=0}^{(k-1)/2} \binom{k+1}{2j} B_{2j} (1 - 2 j)^{\omega(n)}, 
\end{multline*}
 and this may, in turn, and via \eqref{summandmod}, be rewritten as 
\begin{multline*}
 \frac{2n^{k-1}}{k+1} + \frac{2 }{k+1}
 \sum_{j=1}^{(k-1)/2} \binom{k+1}{2j} B_{2j} (z_j L + (1- 2 j)^{\omega(n)}) - \\ 
 \frac{2}{k+1} \sum_{j=0}^{(k-1)/2} \binom{k+1}{2j} B_{2j} (1 - 2 j)^{\omega(n)}, 
\end{multline*}
 for some integers $z_j$ for $j \in \big\{ 1, 2, \ldots, \frac{k-1}{2} \big\}$, so that the above combination of Bernoulli sums reduces to 
\begin{equation}\label{intBernoulli} 
 \frac{2 }{k+1} (n^{k-1}-1) + 
 \sum_{j=1}^{(k-1)/2} \left( \frac{2 }{k+1} \binom{k+1}{2j} B_{2j} L \right) \, z_j. 
\end{equation}
 From our construction, we have that $\operatorname{den}\big( \frac{2}{k+1} \big)$ divides $L$. Recalling \eqref{n1modL}, we find that 
 $n^{k-1}-1$ is a multiple of $L$, so that $\frac{2}{k+1} (n^{k-1}-1)$ is an integer. In a similar spirit, our construction of $L$ is such that 
 each expression of the form $ \frac{2 }{k+1} \binom{k+1}{2j} B_{2j} L$ for $j \in \big\{ 1, 2, \ldots, \frac{k-1}{2} \big\}$
 is an integer. So, since \eqref{intBernoulli} is an integer, 
 we have that $ \frac{ \varphi_k(n) }{ \varphi_1(n) } - C_k(\omega(n))$ is an integer, 
 and, in turn, since $ \frac{ \varphi_k(n) }{ \varphi_1(n) } $ is an integer, 
 we can conclude that $C_{k}(\omega(n))$ is an integer. 

    Observe that  $C_k(u)$ is defined, as in \eqref{defineCku}, independently of the value $n$ given above. 
 Moreover, for an arbitrary integer $v \geq 1$, we can construct an admissible squarefree integer $n$ (in the above manner)
 such that $\omega(n) = v$. So, since $C_{k}(\omega(n)) \in \mathbb{Z}$, we can conclude that 
 $C_{k}(v) \in \mathbb{Z}$ for an arbitrary integer $v \geq 1$. 
 
 So, we have shown that if $k \in \mathcal{D}_{1}$ and $k > 1$, then $k$ is odd and $C_{k}(v) \in \mathbb{Z}$ for all $v \geq 1$. By Lemma 
 \ref{triangularlemma}, we thus have that: For each even integer $h \in [2, k-1]$, 
 we have that 
\begin{equation}\label{2fracBinZ} 
 2 \frac{k!}{h!} B_{h} \in \mathbb{Z}. 
\end{equation}
 
 As above, let $h$ be an even integer such that $h \in [2, k-1]$. 
 Now, let $p$ be an odd prime such that $(p-1) \mid h$. 
 By the von Staudt--Clausen theorem, this same prime $p$
 is a factor in the denominator of $B_{h}$. Moreover, since $p \neq 2$, 
 the factor $2$ in 
 $2 \frac{k!}{h!} B_{h} $ does not cancel with the factor $p$ in the denominator of $B_{h}$. 
 From \eqref{2fracBinZ}, we can conclude that 
 $p$ divides $\frac{k!}{h!} = (h+1)(h+2) \cdots k$. 
 So, we have proved the following property: 
 For an even integer $h \in [2, k-1]$, and for an odd prime $p$
 such that $(p-1) \mid h$, there exists a multiple of $p$ in 
\begin{equation}\label{displayinterval}
 \{ h+1, h+2, \ldots, k \}. 
\end{equation}
 
 Now, let $p \leq k$ denote an odd prime. We write $\rho_p$ in place of the least nonnegative residue of $k$ modulo $p-1$. Since $k$ is 
 odd and $p-1$ is even, 
 it is not the case that $k \equiv 0 \bmod {(p-1)}$, and hence the bounds
\begin{equation}\label{usewithq} 
 1 \leq \rho_p \leq p - 2. 
\end{equation}
 Set $h = k - \rho_p$. Since $k$ is odd and since $\rho_p$ is the residue of an odd number modulo an even number, 
 we see that $h$ is even.
 Moreover, since $\rho_p \equiv k \bmod {(p-1)}$, 
 we find that $(p-1) \mid h$. By assumption
 that $p \leq k$, we see that 
 $h = k - \rho_p \geq k - (p-2) \geq 2$. 
 So, our constructed value $h$ 
 satisfies the above conditions concerning the interval on display in \eqref{displayinterval}, i.e., 
 so that there is at least one value among
\begin{equation}\label{wobrackets}
 k - \rho_p + 1, k - \rho_p+2, \ldots, k 
\end{equation}
 that is a multiple of $p$. 

 Now, let $r_p$ denote the least nonnegative residue of $k$ modulo $p$ (noting the contrast to the definition of $\rho_p$). The greatest 
 multiple of $p$ not exceeding $k$ is $k - r_p$. So, since 
 there is a multiple of $p$ among the 
 values in \eqref{wobrackets}, 
 we can conclude that there exists an index 
\begin{equation}\label{iotain} 
 \iota \in 
 \{ 0, 1, \ldots, \rho_p - 1 \} 
\end{equation} 
 such that 
\begin{equation}\label{kminusrp}
 k - r_{p} = k - \iota. 
\end{equation}
 So, from \eqref{iotain} and \eqref{kminusrp} together, we find that 
 $ r_p < \rho_p$, i.e., so that: For each odd prime $p \leq k$, the relation 
\begin{equation}\label{comparepmod}
 k \pmod {p} < k \pmod {(p-1)} 
\end{equation}
 holds. 

 By way of contradiction, suppose that there exists an odd prime $q$ dividing $k + 1$. Since $k+1$ is even, we see that $q \leq k$. Since 
 $k \equiv -1 \bmod q$, we find that the least nonnegative 
 residue of $k$ modulo $q$ is $q-1$. 
 Being consistent with notation given above, 
 we have that $\rho_q$ 
 is the least nonnegative residue of $k$ modulo $q - 1$, 
 and, from \eqref{usewithq}, we have that 
 $\rho_{q} \leq q-2$. 
 The relation in \eqref{comparepmod} would then give us that 
 $q - 1 < q - 2$, 
 providing a desired contradiction. 

 So, since $k + 1$ does not have any odd prime divisor, we 
 write $k + 1 = 2^{a}$ for some integer $a$, noting that $a \geq 2$ since $k$ is odd and since $ k > 1$. 

 Recall that \eqref{comparepmod} holds for an odd prime $p \leq k$. 
 Since $k > 1$ is odd, we are permitted to let $p = 3$. 
 So, from \eqref{comparepmod} (recalling that $k$ is odd), 
 since $k \bmod {3} < k \bmod {2}$, 
 we can conclude that $3 \mid k$.
 So, since
 $2^{a} \equiv 1 \bmod 3$, we can, in turn, conclude that 
 $a$ is even. 

 By way of contradiction, suppose that $a = 2^{t} u$ for an odd integer $u > 1$ and for $t \geq 1$. 
 We proceed to make use of the property of cyclotomic polynomials
 such that $X^{v} + 1 = (X+1)(X^{v - 1} - X^{v - 2} + X^{v-3} - \cdots - X + 1)$ for an odd, positive integer $v$. 
 This factorization gives us that 
\begin{equation}\label{fromcyclotomic} 
 \big( 2^{2^{t}} + 1 \big) \mid \big( 2^{2^{t} u} + 1 \big), 
\end{equation}
 and we find that \eqref{fromcyclotomic} is equivalent to 
\begin{equation}\label{rewritecyclo}
 \big( 2^{2^{t}} + 1 \big) \mid \left( k + 2 \right). 
\end{equation}
 Now, let $q$ denote an odd prime factor of $2^{2^{t}} + 1$. Recalling that $u > 1$, we find that $q \leq 2^{2^{t}} + 1 < 2^{2^{t} u} - 
 1 = k$. So, since $q$ is an odd prime satisfying $q \leq k$, we have, from \eqref{comparepmod}, that 
\begin{equation}\label{modmodcyclo}
 k \pmod {q} < k \pmod {(q-1)}. 
\end{equation}
 Also, the relation in 
 \eqref{rewritecyclo} gives us that $q \mid (k+2)$, i.e., so that 
\begin{equation}\label{givesqminus2} 
 k \equiv -2 \pmod {q}. 
\end{equation}
 From \eqref{givesqminus2}, we see that the least nonnegative residue of $k$ modulo $q$
 is $q-2$, but the right-hand side of 
 \eqref{modmodcyclo} is in $\{ 0, 1, \ldots, q - 2 \}$, 
 so that the strict inequality in 	 \eqref{modmodcyclo} is contradicted. 

 So, we have shown that $a$ is necessarily a power of $2$. 
 Recalling that $a \geq 2$, we write $a = 2^{\lambda}$ for an integer $\lambda \geq 1$, with 
\begin{equation}\label{kpowertower} 
 k = 2^{2^{\lambda}} - 1.
\end{equation} 
 Since the cases for $\lambda = 1$ and $\lambda = 2$  produce the values $k = 3$ and $k = 15$, respectively,   it remains to consider the  
  cases for $\lambda \geq 3$. 

 We henceforth assume that $\lambda \geq 3$. 

 Now, consider the case whereby $\lambda$ is odd. Recalling that $p \leq k$ denotes an odd prime, and recalling \eqref{kpowertower}, 
 we let $p = 7$. The parity of $2^{\lambda}$ then gives us that $2^{2^{\lambda}} - 1 \equiv 3 \bmod 6$, i.e., so that $k \equiv 3 \bmod 
 6$. The integer sequence given by expressions of the form $2^z \bmod 7$ for positive integers $z$ is $3$-periodic, so that the parity of 
 $\lambda$ gives us that  $2^{2^{\lambda}} \equiv 4 \bmod 7$,  i.e. so that $k \equiv 3 \bmod 7$,   contradicting the strict 
 inequality in \eqref{comparepmod}. 

 So, we have that $\lambda \geq 3$ is even. Now, suppose that $\lambda \equiv 0 \bmod 4$. Again recalling \eqref{kpowertower} and 
 that $p \leq k$ is an odd prime, we proceed to take $p = 11$. The residue class of $2^{\lambda}$ modulo $4$ then gives us that 
 $2^{2^{\lambda}} \equiv 6 \bmod 10$, i.e., so that $k \equiv 5 \bmod 10$. The residue class of $2^{\lambda}$ modulo $10$ then gives 
 us that $2^{2^{\lambda}} \equiv 9 \bmod 11$, i.e., so that $k \equiv 8 \bmod 11$, and we again find that \eqref{comparepmod} 
 is contradicted. 

 From the foregoing arguments, we see that $\lambda \geq 3$ satisfies $\lambda \equiv 2 \bmod 4$, with $\lambda \geq 6$. In this case, 
 we set $p = 29$. Omitting details, the periodicity of powers of $2$ modulo 
 $28$ gives us that $k \equiv 15 \bmod 28$, 
 and the periodicity of powers of $2$ modulo $29$ gives us that 
 $k \bmod 29 \in \{ 15, 23, 24 \}$, 
 and we again find that the strict inequality in \eqref{comparepmod} is contradicted. 

 So, we have shown that it is not the case that $\lambda \geq 3$. So, if $k > 1$, then $k \in \{ 3, 15 \}$, and we can then conclude that 
 the desired containment $\mathcal{D}_{1} \subseteq \{ 1, 3, 15 \}$ holds. 
\end{proof}

 As above, B\"{u}y\"{u}ka\c{s}ik et al.\ \cite[Theorem 2(2)]{BuyukasikGoralSertbas2024} proved that $ \{ 1, 3, 15 \} \subseteq 
 \mathcal{D}_{1}$. So, from the reverse containment proved in Theorem 
 \ref{maintheorem}, we obtain the desired equality 
$ \mathcal{D}_{1} = \{ 1, 3, 15 \}$. 

\section{Conclusion}
 The extension of our above techniques to evaluate $\mathcal{D}_{s}$ for $s > 1$ provides a natural future area of research based 
 on the material in this paper. We encourage the pursuit of this. 

\subsection*{Acknowledgements}
The author acknowledges extensive interactions with 
 GPT-5.5 Pro during the exploratory and proof-development stages of this work. All AI-generated suggestions were substantially revised, corrected, and independently verified by the author, who assumes full responsibility for the mathematical content.

 \end{document}